\RequirePackage{fix-cm}
\documentclass[smallextended]{svjour3}       
\smartqed  
\usepackage{graphicx}
\input{style/packages.sty}
\input{style/ee.sty}
\input{style/commands.sty}

\journalname{arXiv}
\begin{document}

\title{Assessing solution quality in risk-averse stochastic programs
}


\author{E. Ruben van Beesten \and
        Nick W. Koning \and
        David P. Morton
}


\institute{E. Ruben van Beesten \at
              Econometric Institute, Erasmus University Rotterdam, the Netherlands \\
              Faculty of Economics and Management, Norwegian University of Science and Technology, Trondheim, Norway\\
              \email{vanbeesten@ese.eur.nl}            
            \and
            Nick W. Koning \at
            Econometric Institute, Erasmus University Rotterdam, the Netherlands \\
            \email{n.w.koning@ese.eur.nl}            
           \and
            David P. Morton \at
            Department of Industrial Engineering and Management Sciences, Northwestern University, Evanston, IL, United States    \\
            \email{david.morton@northwestern.edu}
}

\date{Received: date / Accepted: date}

\maketitle

\begin{abstract}
    In optimization problems, the quality of a candidate solution can be characterized by the optimality gap. For most stochastic optimization problems, this gap must be statistically estimated. We show that for risk-averse problems, standard estimators are optimistically biased, which compromises the statistical guarantee on the optimality gap. 
    We introduce estimators for risk-averse problems that do not suffer from this bias. Our method relies on using two independent samples, each estimating a different component of the optimality gap. Our approach extends a broad class of optimality gap estimation methods from the risk-neutral case to the risk-averse case, such as the multiple replications procedure and its one- and two-sample variants. 
    We show that our approach is tractable and leads to high-quality optimality gap estimates for spectral and quadrangle risk measures. 
    Our approach can further make use of existing bias and variance reduction techniques.
\keywords{Stochastic programming \and Risk measures \and Optimality gap estimation \and Monte Carlo methods}
\subclass{90C15 \and 91B05} 
\end{abstract}

\newpage


\section{Introduction} \label{sec:introduction}

Consider a risk-averse stochastic program of the form
\begin{align}
    z^* := \min_{x \in \mathcal{X}} \rho(f_x). \label{eq:main}
\end{align}
Here, decision $x$ is constrained to a nonempty, compact set $\mathcal{X} \subseteq \R^d$. The function $f_x(\omega) := f(x, \xi(\omega))$%
\footnote{For example, \eqref{eq:main} is a two-stage recourse model if $f_x$ is of the form $f_x(\omega) = c^T x + \min_{y \in Y} \{ q(\omega)^T y \ : \ T(\omega) x + W(\omega) y \geq h(\omega) \}$, where the first term represents the first-stage costs and the second term is the optimal value of the second-stage problem.}
depends on $x \in \mathcal{X}$ and the random vector $\xi : \Omega \to \R^q$, which is governed by distribution $\Prob$ on the sample space $\Omega$.
The risk measure $\rho$ is a functional on a suitable space $\mathcal{Y}$ of random variables on $\Omega$, which includes $f_x$, for all $x \in \mathcal{X}$. 
We assume $\rho$ and $\mathcal{Y}$ are such that $\rho(f_x)$ is finite for all $x \in \mathcal{X}$, although we will sometimes we require stronger assumptions.

We seek to assess the quality of a candidate solution $\hat{x} \in \mathcal{X}$ to problem~\eqref{eq:main}, e.g., a solution produced by some algorithm.
That is, we want to know whether the \textit{optimality gap}
\begin{align*}
    G = \hat{z} -  z^*,
\end{align*}
between its objective value $\hat{z} := \rho(f_{\hat{x}})$ and the optimal value $z^*$ is small.

The usual approach is to find an upper bound $B \geq G$ on this gap.  If $B$ is small, then the optimality gap $G$ must also be small, so that the solution $\hat{x}$ is of high quality. When the dimension of the random vector $\xi$ is modest and/or we have special structure, such as convexity of the objective with respect to the random vector, there is a literature pursuing deterministic upper bounds~$B$ on~$G$; see the review contained in \cite{narum2023safe}. 

For large-scale stochastic programs that lack such structure, Monte Carlo methods dominate the literature, and we seek a statistical upper bound $B$
\cite{mak1999monte}. In this setting we typically use a sample of $n$ i.i.d.\ draws from the distribution~$\Prob$ and derive a \textit{statistical} estimate $G_n$~of~$G$~\cite{bayraksan2009assessing}. A statistical approximation $B^\alpha$ of the $(1-\alpha)$-quantile of~$G_n$ then serves as a probabilistic upper bound: $B^\alpha \geq G$ with an approximate probability of at least $1-\alpha$ \cite{mak1999monte}.

Unfortunately, this procedure fails in the risk-averse case, as we show in Section~\ref{sec:invalidity}. The reason is that the sample estimator of $\hat{z} = \rho(f_{\hat{x}})$ is downward biased, which sabotages the upward bias of the estimator $G_n$ of $G$. This compromises the guarantee that $B^\alpha \geq G$ with high probability.

In this paper, we address this issue by developing an upward biased estimator of the optimality gap $G$ for a risk-averse stochastic program. We do so for the large class of 
risk measures with minimization representation
\begin{align}
    \rho(Y) = \min_{u \in \mathcal{U}} \E\big[ r(Y, u) \big], \qquad Y \in \mathcal{Y},\label{eq:rho_min}
\end{align}
for some function $r: \R \times \mathcal{U} \to \R$ and closed set $\mathcal{U}$. 
This class includes popular risk measures such as the conditional value-at-risk, entropic risk, and spectral risk, and any risk measure from
the ``expectation risk quadrangles'' of~\cite{rockafellar2013fundamental}; we return to these risk measures in Section~\ref{sec:quality}.
We assume $r(y,\cdot)$ is lower semi-continuous for every $y \in \R$, and we assume $r$ and $\mathcal{U}$ are such that the optimal value of~\eqref{eq:rho_min} is obtained on $\mathcal{U}$ for every $Y \in \mathcal{Y}$. Similarly, we assume $r(f_x,u)$ is lower semi-continuous on $\mathcal{X} \times \mathcal{U}$, w.p.1. With our assumptions on $\mathcal{X}$ and~$\mathcal{Y}$, this suffices to ensure the optimal value of~\eqref{eq:main} is achieved on $\mathcal{X}$. The same will hold for sampling-based counterparts in what follows. 

Section~\ref{sec:valid_estimator} develops our key idea, which uses a second independent ``fresh'' sample of $m$ draws from $\Prob$ to estimate the minimizer of~\eqref{eq:rho_min}. This independent estimator avoids overfitting and thereby produces an estimator $G_{n, m}$ for $G$ that is upward biased. Importantly, our new estimator $G_{n, m}$ allows us to extend a broad class of methods for assessing solution quality in risk-neutral stochastic programs to the risk-averse case, as shown in Section~\ref{sec:conf_interval}. In particular, we give a recipe for deriving probabilistic upper bounds $B^\alpha$ on the optimality gap using the single/two-/multiple replications procedures from~\cite{bayraksan2006assessing,mak1999monte}. 

Section~\ref{sec:quality} discusses the quality of our estimator $G_{n, m}$ for $G$. For example risk measures $\rho$, we illustrate the tractability of computing $G_{n,m}$. 
Moreover, we prove for these examples that the upward bias of $G_{n,m}$ caused by using the fresh sample vanishes as its sample size $m$ grows large.
Finally, Section~\ref{sec:coherent} extends our approach to the class of law invariant coherent risk measures.

\section{Traditional optimality gap estimator} \label{sec:invalidity}

\subsection{Existing literature: risk-neutral case}\label{sec:existing_lit_risk_neutral}

For risk-neutral stochastic programs, i.e., when $\rho = \E$, the optimality gap estimator from the literature \cite{bayraksan2006assessing,mak1999monte} is a plug-in estimator based on an i.i.d. sample $\mathcal{S}_n = (\omega^1, \ldots, \omega^n)$ from $\Prob$. That is, this estimator has form $G_n := \hat{z}_n - z_n^*$, where $\hat{z}_n := \E^n[f_{\hat{x}}]$ is the sample mean estimator of the objective $\hat{z} := \E[f_{\hat{x}}]$ at the candidate solution $\hat{x}$, and $z_n^*:= \min_{x \in \mathcal{X}} \E^n[f_x]$ is the sample estimator of the population optimal value $z^* = \min_{x \in \mathcal{X}} \E[f_{x}]$, also known as its sample average approximation \cite{Kleywegt2002}. In what follows, we allow more general forms of samples $\mathcal{S}_n = (\omega^1, \ldots, \omega^n)$ drawn from $\Prob$. 
Typically, $\hat{z}_n$ and $z_n^*$ are computed using the same sample $\mathcal{S}_n$, which ensures $G_n \ge 0$.

We use notation $\E^n[Y] := \tfrac{1}{n} \sum_{i=1}^n Y(\omega^i)$ to represent the sample mean or empirical expectation with respect to sample $\mathcal{S}_n$. This should not be confused with $\E^{\mathcal{S}_n}[\cdot]$, which denotes the expectation with respect to all possible samples $\mathcal{S}_n$ of size $n$ from distribution $\Prob$. Here, the samples can be i.i.d.\ or sampled in another manner, e.g., to reduce bias or variance. We assume that $\mathcal{S}_n$ satisfies:
\begin{equation}\label{eqn:unbiased}
\E^{\mathcal{S}_n}[\E^n[f_{x}]] = \E[f_{x}], \ \forall x \in \mathcal{X}.
\end{equation}
Note that the $x$ in equation~\eqref{eqn:unbiased} does not depend on $\mathcal{S}_n$ and vice versa.

In the risk-neutral case, $G_n$ is an upward biased estimator of the optimality gap $G$. Specifically, by equation~\eqref{eqn:unbiased},
$\E^{\mathcal{S}_n}[\hat{z}_n] = \E^{\mathcal{S}_n}[\E^n[f_{\hat{x}}]] = \E[f_{\hat{x}}] = \hat{z}$, and  
with $x_n^* \in \argmin_{x \in \mathcal{X}} \E^n[f_x]$ and $x^* \in \argmin_{x \in \mathcal{X}} \E [f_x]$ we have
 $\E^{\mathcal{S}_n}[z_n^*] = \E^{\mathcal{S}_n}[\E^n[f_{x_n^*}]]\leq \E^{\mathcal{S}_n}[\E^n[f_{x^*}]] = \E[f_{x^*}] =  z^*$, where the penultimate equality holds by~\eqref{eqn:unbiased}. As a result,  $\E^{\mathcal{S}_n}[G_n] = \E^{\mathcal{S}_n}[\hat{z}_n] - \E^{\mathcal{S}_n}[z^*_n] \geq G$; see~\cite{mak1999monte,norkin1998branch}.

To derive a probabilistic upper bound $B^\alpha$, we estimate the $(1-\alpha)$-quantile of $G_n$.
One simple approach is to form multiple replications $G_n^1, \ldots, G_n^k$, of $G_n$ and choose $B^\alpha$ equal to the $(1-\alpha)$-quantile of the associated Student's $t$ random variable. 
By the central limit theorem, this yields the approximate probabilistic upper bound on $G$, $\mathbb{P}(B^\alpha \geq G) \gtrapprox 1 - \alpha$, for sufficiently large~$k$. This is the multiple replication procedure of \cite{mak1999monte}.
More advanced versions based on only one or two replications of $G_n$ also exist \cite{bayraksan2006assessing}. These procedures are now common practice in the risk-neutral case. Various authors have further developed the methodology by reducing the bias and/or variance of the optimality gap estimator, extending the procedure to multi-stage problems, and otherwise improving the method, e.g.,
\cite{%
chen2021confidence,%
chen2023software,%
chen2024distributions,%
chiralaksanakul2004assessing,%
de2017assessing,%
lam2017empirical,%
linderoth2006empirical,%
love2011overlapping,%
partani2006jackknife,%
stockbridge2016variance}.

\subsection{Risk-averse case: invalid upper bound due to a downward bias}\label{sec:risk_averse_invalid}

We now show that with risk measures of form~\eqref{eq:rho_min}, the approach just sketched does not yield a valid probabilistic upper bound $B^\alpha$. Analogous to the risk-neutral case, we define the sample estimator $G_n$ of the optimality gap $G$ as
\begin{align}
    G_n := \hat{z}_n - z_n^*,
\end{align}
where $\hat{z}_n := \rho^n( f_{\hat{x}})$ and $z_n^* := \min_{x \in \mathcal{X}} \rho^n(f_x)$ estimate $\hat{z}=\rho(f_{\hat{x}})$~and~$z^*$. 

The validity of $B^\alpha$ breaks down in the risk-averse case, as the estimator~$G_n$ is not guaranteed to be upward biased for $G$ (Proposition~\ref{prop:ineqs}), which is caused by a downward bias of the sample estimator $\rho^n$ of the risk measure $\rho$ (Lemma~\ref{lemma:bias}), defined as
\begin{equation}\label{eqn:empirical_risk_Y}
    \rho^n(Y) := \min_{u \in \mathcal{U}} \E^n[r(Y,u)], \qquad Y \in \mathcal{Y}.
\end{equation}
The latter bias occurs because the minimizer in $\mathcal{U}$ ``overfits'' to the sample.
\begin{lemma}[Bias of $\rho^n$] \label{lemma:bias}
 Let $\rho$ be a risk measure of the form \eqref{eq:rho_min}, let $\rho^n$ be its empirical analog~\eqref{eqn:empirical_risk_Y} with respect to $\mathcal{S}_n$, and assume $\mathcal{S}_n$ satisfies \linebreak $\E^{\mathcal{S}_n}[\E^n [r(Y,u)]] = \E[r(Y,u)], \ \forall u \in \mathcal{U}$. Then, for all $Y \in\mathcal{Y}$,
    \begin{align*}
        \E^{\mathcal{S}_n} \left[\rho^n(Y) \right] \leq \rho(Y).
    \end{align*}
\end{lemma}
\begin{proof}
    Let $u_n^* \in \argmin_{u \in \mathcal{U}} \E^n[r(Y,u)]$ and $u^* \in \argmin_{u \in \mathcal{U}} \E [r(Y,u)]$. Then,
    \begin{align*}
        \E^{\mathcal{S}_n} \left[\rho^n(Y) \right] &= \E^{\mathcal{S}_n} \left[ \E^n\big[ r(Y, u_n^*) \big] \right] \leq \E^{\mathcal{S}_n} \left[ \E^n\big[ r(Y, u^*) \big] \right] 
        = \E\big[  r(Y, u^*) \big] = \rho(Y).
    \end{align*}
    \qed
\end{proof}

\begin{proposition}[Bias of $z^*$ and $\hat{z}_n$] \label{prop:ineqs}
Let $\rho$ be a risk measure of the form \eqref{eq:rho_min}, let $\rho^n$ be its empirical analog~\eqref{eqn:empirical_risk_Y}, and assume $\mathcal{S}_n$ satisfies \begin{equation}\label{eqn:unbiased_risk}
    \E^{\mathcal{S}_n}[\E^n [r(f_x,u)]] = \E[r(f_x,u)], \ \forall u \in \mathcal{U}, x \in \mathcal{X}.
\end{equation}  
    Let $\hat{x} \in \mathcal{X}$ be used to define $\hat{z}_n = \rho^n( f_{\hat{x}})$, and let $z_n^* = \min_{x \in \mathcal{X}} \rho^n(f_x)$.
    Then,
    \begin{align}
        \E^{\mathcal{S}_n} \left[ z^*_n \right] \leq z^* \quad \text{and} \quad \E^{\mathcal{S}_n} \left[ \hat{z}_n \right] \leq \hat{z}, \label{eq:ineqs}
    \end{align}
where $z^*$ is the optimal value of~\eqref{eq:main} and $\hat{z}=\rho(f_{\hat{x}})$.
\end{proposition}
\begin{proof}
    Using Lemma~\ref{lemma:bias}, $\E^{\mathcal{S}_n} \left[ z^*_n \right] = \E^{\mathcal{S}_n} \left [ \E^n [ r(f_{x_n^*},u_n^*) ] \right ]  \leq  \E^{\mathcal{S}_n} \left[ \E^n[r(f_{x^*},u^*)] \right] = \min_{x \in \mathcal{X}} \rho(f_x) = z^*$. The second statement follows similarly, with $x_n^*$ and $x^*$ replaced by $\hat{x}$. \qed
\end{proof}
When $\rho = \E$, the right inequality in \eqref{eq:ineqs} holds with equality \cite{mak1999monte}, so 
$
\E^{\mathcal{S}_n} \left[ G_n \right] = \E^{\mathcal{S}_n} \left[ \hat{z}_n - z^*_n \right] \geq \hat{z} - z^* = G,
$ i.e., $G_n$ is an upward biased estimator of~$G$ in the risk-neutral case. In general, however, the right inequality in \eqref{eq:ineqs} need not hold with equality, in which case upward bias of $G_n$ cannot be guaranteed.

\section{Conservative optimality gap estimator} \label{sec:valid_estimator}

To address the issue in Section \ref{sec:risk_averse_invalid}, we derive a conservative, i.e., upward biased estimator of $G$ for a risk-averse stochastic program. We start with Proposition~\ref{prop:ineqs}'s observation that $\hat{z}_n$ is downward biased because in
$
\rho^n(f_{\hat{x}}) = \min_{u \in \mathcal{U}} \E^n\left[ r(f_{\hat{x}}, u) \right],
$
the optimal solution $u^*_n$ ``overfits'' to the sample $\mathcal{S}_n$. This suggests reversing the bias by using another $\hat{u} \in \mathcal{U}$ that does not depend on~$\mathcal{S}_n$. Indeed, any fixed $\hat{u} \in \mathcal{U}$ yields an upward biased estimator $\E^n\big[ r(f_{\hat{x}},\hat{u})\big]$ of~$\rho(f_{\hat{x}})$:
\begin{align}
    \E^{\mathcal{S}_n} \left[ \E^n\big[ r(f_{\hat{x}},\hat{u}) \big] \right]= \E\big[ r(f_{\hat{x}}, \hat{u}) \big] \geq \min_{u \in \mathcal{U}} \E\big[ r(f_{\hat{x}}, u) \big] = \rho(f_{\hat{x}}). \label{eq:u_bar_ineq}
\end{align}

If $\hat{u} = u^* \in \argmin_{u \in \mathcal{U}} \E\big[ r(f_{\hat{x}}, u)\big]$ then \eqref{eq:u_bar_ineq}'s inequality holds with equality, so $\E^n\big[ r(f_{\hat{x}},\hat{u}) \big]$ is an unbiased estimator of $\rho(f_{\hat{x}})$.  Of course, computing $u^*$ is typically impossible.

Instead, we estimate $u^*$ by drawing a second {\it fresh sample} $\mathcal{S}_m = (\tilde{\omega}^1, \ldots, \tilde{\omega}^m)$ from $\Prob$, independent of $\mathcal{S}_n$. Then, we replace $\hat{u}$ by the sample estimator\footnote{If $\rho$ comes from a risk quadrangle \cite{rockafellar2013fundamental}, then $u^*$ is called the \textit{statistic} associated with $\rho$. In that terminology, our approach is to estimate this statistic by the ``sample statistic'' $u^*_m$.} 
to form the two-sample estimator $G_{n,m}$ as follows:
\begin{subequations}\label{eqn:two_sample_estimator}
\begin{align}
    u^*_m & \in  \argmin_{u \in \mathcal{U}}, 
    {\E}^m\big[ r(f_{\hat{x}}, u) \big] \label{eqn:two_sample_estimator_a}, \\
    \hat{z}_{n,m} & :=  \E^n[r(f_{\hat{x}}, u^*_m)] \label{eq:z_hat_nm}, \\
    z^*_n & := \min_{x \in \mathcal{X}} \rho^n (f_x), \label{eqn:two_sample_estimator_c} \\
    G_{n,m} & := \hat{z}_{n,m} - z^*_n,
\end{align}
\end{subequations}
where ${\E}^m$ denotes $\mathcal{S}_m$'s  empirical expectation. The motivation is that for $m$ large enough  $u^*_m$'s quality is high and the upperward bias in ~\eqref{eq:z_hat_nm} is low.

Optimization over $u$ in~\eqref{eqn:two_sample_estimator_a} is typically easy relative to the optimization over $x$ in~\eqref{eqn:two_sample_estimator_c}. Hence, we can choose $m \gg n$. Indeed, the optimization over $u$ is often univariate, and for the risk measures we consider in Section~\ref{sec:quality} the effort is cheaper than sorting. 
%
%
The following theorem establishes that $G_{n,m}$ is an upward biased estimator of $G$; see Figure~\ref{fig:plot_4_full} for an illustration.

\begin{theorem}[Upward bias of $G_{n,m}$]\label{theorem:upperward_bias}
Assume the hypotheses of Proposition~\ref{prop:ineqs} hold. Further, let $\mathcal{S}_m$ satisfy~\eqref{eqn:unbiased_risk} and be independent of $\mathcal{S}_n$. Let $u_m^*$, $\hat{z}_{n,m}$, and $G_{n,m}$ be defined as in~\eqref{eqn:two_sample_estimator}.
%
Then,
    \begin{align*}
        \E^{\mathcal{S}_n} \left[ \hat{z}_{n,m} \right] \geq \hat{z}, \qquad \text{ and } \qquad \E^{\mathcal{S}_n} \left[ G_{n,m} \right] \geq G,
    \end{align*}
    so $G_{n,m}$ is an upward biased estimator of $G$.
\end{theorem}
\begin{proof}
    The first inequality follows from substituting $\hat{u} = u^*_m$ in \eqref{eq:u_bar_ineq}. Combining this with $\E^{\mathcal{S}_n} \left[ z^*_n \right] \leq z^*$ from Proposition~\ref{prop:ineqs} yields the second inequality. \qed
\end{proof}

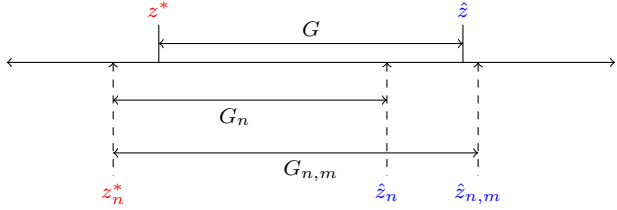
\begin{figure}
    \centering
    \centering
\begin{tikzpicture}   
    \draw[<->] (-4,0) -- (4,0);                                                 
    \draw (-2,0) -- + (0,0.5) node[above] {\red{$z^*$}};                        
    \draw (2,0) -- + (0,0.5) node[above] {\blue{$\hat{z}$}};                    
    \draw[<->] (-2,0.25) -- (2,0.25);                                           
    \draw (0,0.25) node[above] {$G$};                                           
    \draw[<-, dashed] (-2.6,0) -- + (0,-1.5) node[below] {\red{$z^*_n$}};           
    \draw[<-, dashed] (1,0) -- + (0,-1.5) node[below] {\blue{$\hat{z}_n$}};       
    \draw[<-, dashed] (2.2,0) -- + (0,-1.5) node[below] {\blue{$\hat{z}_{n,m}$}};
    \draw[<->] (-2.6,-0.5) -- (1,-0.5);                                           
    \draw (-1,-0.5) node[below] {$G_n$};                                        
    \draw[<->] (-2.6,-1.2) -- (2.2,-1.2);                                        
    \draw (0,-1.2) node[below] {$G_{n,m}$};                                   
\end{tikzpicture}
    \caption{Illustration of the optimality gap and its estimators. Dashed arrows indicate the expectation of the indicated random object.}
    \label{fig:plot_4_full}
\end{figure}


\section{A valid probabilistic upper bound on the optimality gap} \label{sec:conf_interval}

Based on the new optimality gap estimator $G_{n, m}$ we now derive a probabilistic upper bound $B^\alpha$ such that $\Prob(G \leq B^\alpha) \gtrapprox 1 - \alpha$. In doing so, we show how existing methods for the risk-neutral case extend to the risk-averse case. 

The key idea is to frame the risk-averse problem \eqref{eq:main} as a risk-neutral problem in a higher-dimensional space, as done in Proposition~\ref{prop:ineqs}'s proof.
In particular, consider problem~\eqref{eq:main} with $\rho$ as in~\eqref{eq:rho_min}. Then with $g_{x,u} := r(f_x, u)$,
\begin{align}
    z^* = \min_{x \in \mathcal{X}} \rho(f_x) = \min_{x \in \mathcal{X}} \min_{u \in \mathcal{U}} \E \big[ r(f_x, u) \big] = \min_{x \in \mathcal{X}, u \in \mathcal{U}} \E[g_{x,u}]. \label{eq:reformulation}
\end{align} 

This reformulation could be used to compute a candidate pair $(\hat{x},\hat{u})$ by solving $\min_{x \in \mathcal{X}, u \in \mathcal{U}} \E^{n_0}[g_{x,u}]$ under yet another independent sample. However, even if we compute $\hat{x} \in \mathcal{X}$ in this way (our approach allows for alternatives), we recommend {\it not} using such a $\hat{u}$. Because $\min_{u \in {\mathcal U}} \E^m[g_{\hat{x}, u}]$ is typically much easier to solve than when we simultaneously optimize over $x$ and $u$, we can choose $m \gg n_0$, which improves the quality of $\hat{u}$ and reduces associated bias.



Thus given $\hat{x} \in \mathcal{X}$, we compute $u_m^*$ with $\mathcal{S}_m$ using~\eqref{eqn:two_sample_estimator_a}. Then, using $(\hat{x},u_m^*)$ as the candidate solution to~\eqref{eq:reformulation} 
we apply standard risk-neutral theory for estimating its optimality gap via 
$
G_{n,m} = \E^n[g_{\hat{x}, u^*_m}] - \min_{x \in \mathcal{X}, u \in \mathcal{U}} \E^n[g_{x,u}],
$
which like the true optimality gap is non-negative.
%
This leads to the following recipe for producing a probabilistic upper bound in the risk-averse case. 

\paragraph{Recipe for deriving an (approximate) probabilistic upper bound on $G$:}
\begin{enumerate}
    \item Take $\hat{x} \in \mathcal{X}$ as input.
    \item Reformulate the risk-averse problem~\eqref{eq:main} as a risk-neutral problem~\eqref{eq:reformulation} using $y=(x,u)$ as the decision vector.
    \item Form sample $\mathcal{S}_m$ and compute $u_m^*$ via~\eqref{eqn:two_sample_estimator_a} to yield $\hat{y}=(\hat{x}, u^*_m)$ as a candidate solution to reformulation~\eqref{eq:reformulation}.
    \item Using samples that are independent of $\mathcal{S}_m$, apply any risk-neutral method that produces an (approximate) probabilistic upper bound $B^\alpha$ for the optimality gap of the reformulated problem at $\hat{y}$. \label{final_step}
\end{enumerate}
The resulting $B^\alpha$ is an (approximate) probabilistic upper bound for the optimality gap of the risk-averse problem \eqref{eq:main} at $\hat{x}$.

\begin{remark}
    An attractive feature of the way we form the recipe is that we do not need to re-derive assumptions in the risk-averse setting to ensure the procedure is valid. Rather, we simply need to ensure that the assumptions required by the risk-neutral method used in step~\ref{final_step} are satisfied by the reformulated problem \eqref{eq:reformulation} with decision vector $(x,u) \in \mathcal{X} \times \mathcal{U}$ and objective function $g_{x,u}$. 
    For example, the multiple replications procedure in \cite{mak1999monte} requires that $g_{x,u}$ has a finite mean and variance for each $x \in \mathcal{X}$ and $u \in \mathcal{U}$. The single- and two-replication procedures from \cite{bayraksan2006assessing} require stronger assumptions: $\mathcal{X} \times \mathcal{U}$ is compact, $\E \left[ \sup_{x \in \mathcal{X}, \, u \in \mathcal{U}} (g_{x,u})^2 \right] < \infty$, and $g_{x,u}$ is continuous in $(x,u)$, w.p.1. It is natural that $\mathcal{X}$ is bounded for decision problems, but typically $\mathcal{U}$ is unbounded, e.g., $\mathcal{U}=\R$. For specific functions $g_{x,u}=r(f_x,u)$, and underlying random vectors $\xi$, bounds for $u^*$ can be pre-computed and $\mathcal{U}$ altered accordingly. 
\end{remark}

\begin{remark}
The multiple replications procedure, as well as the single- and two-replication procedures, can form
$
G_{n,m} = \E^n[g_{\hat{x}, u^*_m}] - \min_{x \in \mathcal{X}, u \in \mathcal{U}} \E^n[g_{x,u}]
$
using the same samples $\mathcal{S}_n$ for both terms. This ensures that, like the true optimality gap, the estimator is non-negative. Moreover, using common random numbers in this way tends to reduce the variance of the resulting estimator.  
\end{remark}

\begin{remark}
Our recipe is {\it not} designed for settings in which generating a large number of samples is computationally expensive or impossible. Rather in our context, generating Monte Carlo samples (i.e., $\mathcal{S}_m$ and $\mathcal{S}_n$) is computationally cheap and, as noted above, optimizing over $u$ for fixed $x$ is computationally cheap, all relative to optimizing over $x \in \mathcal{X}$ and optimizing over $(x,u) \in \mathcal{X} \times \mathcal{U}$.
\end{remark}

The recipe's output $B^\alpha$ is an (approximate) probabilistic upper bound on the optimality gap $G$ of the risk-averse problem~\eqref{eq:main} at $\hat{x}$. Concretely, our recipe extends the multiple replications procedure from \cite{mak1999monte} and its one- and two-sample variants from \cite{bayraksan2006assessing} to the risk-averse case, and in step~\ref{final_step} it can make use of bias and variance reduction techniques referenced in Section~\ref{sec:invalidity}.

\section{Tractability and quality of $\hat{z}_{n,m}$} \label{sec:quality}

Compared to the estimator $G_n = \hat{z}_n - z^*_n$ from the risk-neutral literature, our estimator $G_{n,m} = \hat{z}_{n,m} - z^*_n$ uses the same estimator $z^*_n$ of $z^*$, but a different estimator of $\hat{z}$, namely $\hat{z}_{n,m}$. Properties of the estimator $z^*_n$ and of $\hat{z}_n$ (when $\hat{x}=x_n^*$ is computed by a sample problem) have been widely studied in the literature.
In this section, we investigate the tractability and quality---in terms of bias---of~$\hat{z}_{n,m}$. We do so for two subclasses of minimization risk measures: expectation quadrangle risk measures and spectral risk measures.

\subsection{Expectation quadrangle risk measures} \label{subsec:quadrangle}

We first discuss risk measures from an expectation risk quadrangle~\cite{rockafellar2013fundamental}, which have form $\rho(Y) = \min_{u \in \mathcal{U}} \E[r(Y,u)]$, with $\mathcal{U} = \R$ and $r(Y,u) = u + v(Y - u)$, for some convex function $v : \R \to \R$. The quadrangle is called \textit{regular} if $v(0) = 0$ and $v(x) > x$ for all $x \neq 0$. Examples include conditional value-at-risk (CVaR), with $v(t) = (1-\alpha)^{-1} \max\{t, 0\}$ for parameter $\alpha \in [0,1)$, and the entropic risk measure (ENT), with $v(t) = \theta^{-1} (e^{\theta t} - 1)$ for parameter $\theta > 0$.

\subsubsection{Tractability}

To assess tractability of $\hat{z}_{n,m}$, note that we can write
\begin{align*}
    \hat{z}_{n,m} = \E^n\left[ r(f_{\hat{x}}, u^*_m) \right] = \E^n\left[ u^*_m + v(f_{\hat{x}} - u^*_m) \right] =  u^*_m  + \frac{1}{n} \sum_{i=1}^n v(f_{\hat{x}}(\omega^i) - u^*_m).
\end{align*}
This is tractable provided $f_{\hat{x}}(\omega)$ and $u^*_m$ are tractable. The former is requisite to our approach, and the latter holds because it involves univariate optimization over $u$ of a convex function, at least when the measures are {\it regular} \cite{rockafellar2013fundamental}. For some cases we effectively have analytical solutions. 
\begin{example}[CVaR: tractability] \label{ex:cvar_tractability}
    For CVaR, $u^*_m$ is the sample quantile $q^m_\alpha(f_{\hat{x}})$, i.e., the 
    $\ceil{\alpha m}$th 
order statistic of $f_{\hat{x}}(\tilde{\omega}^1), \ldots, f_{\hat{x}}(\tilde{\omega}^m)$.
\end{example}
\begin{example}[ENT: tractability] \label{ex:ent_tractability}
    For the entropic risk measure, $u^*_m$ is the sample risk measure itself
${\ent}_\theta^m(f_{\hat{x}})
= \theta^{-1} \log 
{\E}^m
\left[ e^{\theta f_{\hat{x}}}\right] = \theta^{-1} \log \left( \frac{1}{m} \sum_{j=1}^m e^{\theta f_{\hat{x}}(\tilde{\omega}^j)} \right)$.
\end{example}

\subsubsection{Quality}

The following result provides a quality guarantee for $\hat{z}_{n,m}$: we show that under a regularity condition on the expectation quadrangle risk measure 
the bias of $\hat{z}_{n,m}$ vanishes as $m \to \infty$.

\begin{proposition}[Asymptotic unbiasedness of $\hat{z}_{n,m}$, quadrangle]\label{prop:quadrangle_convergence}
Assume the hypotheses of Theorem~\ref{theorem:upperward_bias} hold. Further,
    let risk measure $\rho(f_{\hat{x}})= \min_u \E \left[ u + v(f_{\hat{x}} - u) \right]$ be from a regular expectation risk quadrangle. Suppose that $\E[v(f_{\hat{x}})] < \infty$ and $\E\big[| \delta(f_{\hat{x}}) | \big] < \infty$ for some subgradient $\delta(f_{\hat{x}})$ of $u \mapsto v(f_{\hat{x}} - u)$ at $u=0$. Finally, suppose that the level sets of $\E \left[ r(f_{\hat{x}}, \cdot) \right]$ are bounded.
Then, $\lim_{m \to \infty} \E^{\mathcal{S}_n}\left[ \hat{z}_{n,m} \right] = \hat{z}$ with probability one.
\end{proposition}
\begin{proof}
    The result follows from Theorem~3.4 in \cite{king1991epi}; we first check that its conditions are satisfied. 
    Clearly, $\mathcal{U} = \R$ is a reflexive Banach space with separable dual $\R$. Furthermore, for $\bar{u} = 0$, we have $\E[\bar{u} + v(f_{\hat{x}} - \bar{u})] = \E[v(f_{\hat{x}})] < \infty$ by assumption. Moreover, $1 - \delta(f_{\hat{x}}) \in \partial_u r(f_{\hat{x}},\bar{u})$ satisfies $\E| 1 - \delta(f_{\hat{x}}) | \leq 1 + \E\big[|\delta(f_{\hat{x}})|\big] < \infty$ by assumption. The remaining conditions follow from the fact that $v$ is continuous since it is convex and finite-valued on $\R$. Thus, by Theorem~3.4 in \cite{king1991epi} it follows that with probability one, every cluster point of the sequence $\{ u_m^*\}_{m=1}^\infty$ solves $\min_{u \in \mathcal{U}} \E\left[ r(f_{\hat{x}}, u) \right] = \hat{z}$. 
Because the level sets of $\E \left[ r(f_{\hat{x}}, \cdot) \right]$ are bounded we know
that, with probability one, at least one cluster point of $\{ u_m^*\}_{m=1}^\infty$ exists; see Exercise 7.32 and Theorem 7.33 of \cite{rockafellar1998variational}. 
By continuity of $r$ in $u$ it follows that with probability one, the sequence with elements $\E^{\mathcal{S}_n}\left[ \hat{z}_{n,m} \right] = \E^{\mathcal{S}_n}\left[ \E^n\left[ r(f_{\hat{x}}, u^*_m) \right] \right] = \E\left[ r(f_{\hat{x}}, u^*_m) \right]$ has a unique cluster point $\hat{z}$, i.e.,  $\lim_{m \to \infty} \E^{\mathcal{S}_n}\left[ \hat{z}_{n,m} \right] = \hat{z}$. \qed
\end{proof}

\begin{example}[CVaR: quality] \label{ex:cvar_quality}
    Conditional value-at-risk ($\cvar_\alpha$), $\alpha \in [0,1)$, is from the regular expectation risk quadrangle with $v(t) = (1-\alpha)^{-1} \max\{t, 0\}$. We have $\E[v(f_{\hat{x}})] = (1-\alpha)^{-1} \E[\max\{f_{\hat{x}}, 0\}] < \infty$ because $\E[f_{\hat{x}}]$ is finite, as it is upper bounded by $\cvar_\alpha(f_{\hat{x}}) < \infty$. 
    Furthermore, any subgradient $\delta$ of $v$ satisfies $\E\big[ | \delta | \big] \leq \E\big[ (1-\alpha)^{-1} ] < \infty$. 
    Finally, it is easy to verify that level sets of $\E \left[ r(f_{\hat{x}}, \cdot) \right]$ are bounded.
    Thus, Proposition~\ref{prop:quadrangle_convergence} applies.
\end{example}
\begin{example}[ENT: quality]
    Entropic risk measure, $\ent_\theta(Y) = \theta^{-1} \log \E\left[ e^{\theta Y} \right]$ with $\theta > 0$, is from the regular expectation risk quadrangle with $v(t) = \theta^{-1}(e^{\theta t} - 1)$. 
    As $\ent_\theta(f_{\hat{x}}) = \theta^{-1} \log \E[e^{\theta f_{\hat{x}}}] < \infty$, the expectation $\E[e^{\theta f_{\hat{x}}}]$ is finite. Thus, we have $\E[v(f_{\hat{x}})] = \theta^{-1}(\E[e^{\theta f_{\hat{x}}}] - 1) \leq \theta^{-1}\E[e^{\theta f_{\hat{x}}}] < \infty$.
    Moreover, as $v$ is differentiable, $\delta(f_{\hat{x}}) = v'(f_{\hat{x}}) = e^{\theta f_{\hat{x}}}$, and thus, $\E\big[ | \delta(f_{\hat{x}})| \big] = \E\big[ e^{\theta f_{\hat{x}}} \big] < \infty$. Finally, it is again easy to verify that level sets of $\E \left[ r(f_{\hat{x}}, \cdot) \right]$ are bounded.
    Thus, Proposition~\ref{prop:quadrangle_convergence} applies.    
\end{example}

\subsection{Spectral risk measures}

Next, we discuss spectral risk measures \cite{acerbi2002spectral}, which have the form
\begin{align}
    \rho(Y) = \int_{[0,1)} \cvar_\alpha(Y) d\mu(\alpha), \qquad Y \in\mathcal{Y}, \label{eq:spectral}
\end{align}
for some probability measure $\mu$ on the interval $[0,1)$, called the Kusuoka representer of $\rho$.
%
Before discussing quality and tractability of $\hat{z}_{n,m}$, we show that spectral risk measures are of the form~\eqref{eq:rho_min}.

\begin{lemma}[Spectral risk measures are of form \eqref{eq:rho_min}]
    Let $\rho$ be a spectral risk measure defined on the space of 
    non-negative 
    random variables in $\mathcal{Y}$. Then, $\rho$ is of the form \eqref{eq:rho_min} with $\mathcal{U} = \{ u : [0,1) \to \R_+\}$ and $r(z,u) = \int_{[0,1)} [ u(\alpha) + (1-\alpha)^{-1}(z - u(\alpha))^+ ] d\mu(\alpha)$.
\end{lemma}
\begin{proof}
    First, 
    $\rho(Y) = \min_{u \in \mathcal{U}} \int_{[0,1)} \E\left[u(\alpha) + (1-\alpha)^{-1} (Y - u(\alpha))^+ \right] d\mu(\alpha)$ by Proposition~3.2 of \cite{acerbi2002portfolio}. By non-negativity of $Y$ and Tonelli's theorem, we can swap the integral and expectation, yielding $\rho(Y) = \min_{u \in \mathcal{U}} \E\left[ r(Y, u) \right]$. 
    \qed
\end{proof}

\begin{remark}
    In most applications, the restriction $Y \geq 0$ is without loss, as $f_x \geq 0$ for all $x \in \mathcal{X}$ is typically satisfied, possibly after translation with a constant. 
\end{remark}

\subsubsection{Tractability} \label{subsubsec:spectral_tractability}

For tractability of $\hat{z}_{n,m}$, first consider $u^*_m$. From Example~\ref{ex:cvar_tractability}, we know that $u^*_m(\alpha)$ is the $\ceil{\alpha m}$th order statistic of $f_{\hat{x}}(\tilde{\omega}^1), \ldots, f_{\hat{x}}(\tilde{\omega}^m)$, $\alpha \in [0,1)$. Thus, $u^*_m(\alpha)$ is constant on each interval $\left( \tfrac{k-1}{m}, \frac{k}{m} \right]$, $k=1,\ldots,m$. Thus, the integral
$
    \hat{z}_{n,m} = \int_{[0,1)} \left( u^*_m(\alpha) + (1-\alpha)^{-1}\E^n\left[ (f_{\hat{x}} - u^*_m(\alpha))^+ \right]  \right) d\mu(\alpha)
$
collapses to two sums. The first involves coefficients $\mu_0, \ldots, \mu_m$, where $\mu_0 = \mu(\{0\})$ and $\mu_k =  \mu\left( (\tfrac{k-1}{m}, \tfrac{k}{m}] \right)$, $k=1,\ldots,m$. The second involves coefficients 
$\bar{\mu}_0 = \mu(\{0\})$ and $\bar{\mu}_k = \int_{\left(\frac{k-1}{m}\right)^+}^{\frac{k}{m}} (1-\alpha)^{-1} d\mu\left(\alpha \right)$, $k=1,\ldots,m$. Hence, $\hat{z}_{n,m}$ is easy to compute.

\subsubsection{Quality}

Under a regularity condition on the spectral risk measure $\rho$, 
we can show that the bias of $\hat{z}_{n,m}$ vanishes as $m \to \infty$. 
The restriction to non-negative $f_{\hat{x}}$ is without loss in most applications.

\begin{proposition}[Asymptotic unbiasedness of $\hat{z}_{n,m}$, spectral]
    Let $\rho$ be a spectral risk measure defined on the space of non-negative random variables in~$\mathcal{Y}$
    whose Kusuoka representer $\mu$ has a finite support.
    Moreover, assume $f_{\hat{x}}$ is non-negative.
    Then, $\lim_{m \to \infty} \E^{\mathcal{S}_n}\left[ \hat{z}_{n,m} \right] = \hat{z}$ with probability one.
\end{proposition}
\begin{proof}
    The result follows from Theorem~3.4 in \cite{king1991epi}; we first check that its conditions are satisfied.
    An optimal solution of $\min_{u \in \mathcal{U}} \E\left[ r(f_{\hat{x}}, u) \right] = \hat{z}$ is given by $u^*(\alpha) = q_{\alpha}(f_{\hat{x}})$, $\alpha \in [0,1)$, where $q_{\alpha}(f_{\hat{x}})$ is the $\alpha$-quantile of $f_{\hat{x}}$. 
    Writing $\bar{\alpha} := \max\{\alpha \in [0,1) \ : \ \mu(\{\alpha\}) > 0\}$, we have $\int_{[0,1)} (u^*(\alpha))^2 d\mu(\alpha) \leq \big(q_{\bar{\alpha}}(f_{\hat{x}})\big)^2 < \infty$, so $u^* \in L_2([0,1), \mu)$. 
    So without loss, we can restrict $\mathcal{U}$ to the reflexive Banach space $L_2([0,1), \mu)$ with separable dual space $L_2([0,1), \mu)$. Moreover, for $\bar{u} = 0$, every subgradient $\delta(\omega)$ of $r(\bar{u}, \omega)$ with respect to $u(\alpha)$ falls within the interval $[-\frac{\alpha}{1-\alpha}, 1]$, which is bounded in absolute value by $(1-\alpha)^{-1}$. Thus, $\int_\Omega \int_{[0,1)} \delta(\omega)^2 d\mu(\alpha) d\Prob(\omega) \leq \int_{[0,1)} (1-\alpha)^{-2} d\mu(\alpha) 
    \leq (1 - \bar{\alpha})^{-2} < \infty$. 
    The remaining conditions follow from convexity and continuity of $r$ in $u$.
    Thus, Theorem~3.4 in \cite{king1991epi} applies. 
    Moreover, by Example~\ref{ex:cvar_quality} and the proof of Proposition~\ref{prop:quadrangle_convergence}, $\{u_m^*(\alpha)\}_{m=1}^\infty$ has a cluster point for each $\alpha \in \text{support}(\mu)$, w.p.1. Diagonalization over $\alpha \in \text{support}(\mu)$ yields a subsequence of $\{u^*_m\}_{m=1}^\infty$ with a cluster point, w.p.1.
    The result now follows analogously as in Proposition~\ref{prop:quadrangle_convergence}. \qed
\end{proof}

\section{Extension to law invariant coherent risk measures} \label{sec:coherent}

We conclude with an extension of our approach to law invariant coherent risk measures, which can be expressed as
\begin{align}\label{eqn:licrm}
    \rho(Y) = \sup_{\mu \in \mathcal{M}} \int_{[0,1)} \cvar_\alpha(Y) d\mu(\alpha), \qquad Y \in\mathcal{Y},
\end{align}
and generalize spectral risk measures. We restrict attention to the space $\mathcal{Y}$ of bounded random variables $Y$. Here, $\mathcal{M}$ is a closed convex set of probability measures on $[0,1)$ that does not depend on the probability measure $\Prob$ \cite{dentcheva2020risk}. 

For these risk measures, the direction of the bias of a nominal sample estimator $\rho^n$ of $\rho$ is not clear a priori. On the one hand, overfitting the optimizer in the minimization representation of each $\cvar_\alpha$ to the sample distribution leads to a negative bias. While this would be problematic for an estimator of $\hat{z}$, it can be solved in the exact same way as for spectral risk measures, i.e., by using $u^*_m(\alpha)$, $\alpha \in [0,1)$, based on a second sample. This leads to the estimator
\begin{align}
    \hat{z}_{n,m} := \sup_{\mu \in \mathcal{M}} \int_{[0,1)} \left( u^*_m(\alpha) + (1 - \alpha)^{-1} \E^n\left[ (f_{\hat{x}} - u^*_m(\alpha))^+ \right] \right) d\mu(\alpha), \label{eq:z_hat_nm_coherent}
\end{align}
satisfying $\E^{\mathcal{S}_n}\big[ \hat{z}_{n,m}\big] \geq \hat{z}$.

On the other hand, overfitting the optimizer over $\mu \in \mathcal{M}$ to the sample distribution leads to an upward bias. This is problematic for the sample estimator $\hat{z}^*_n$ of $z^*$, which is now not guaranteed to have a downward bias. To solve this, we propose an analogous approach and estimate the optimal $\mu \in \mathcal{M}$ using a \textit{third} sample $\mathcal{S}_{\ell}$, independent of $\mathcal{S}_n$ (but not necessarily $\mathcal{S}_m$). Specifically, for
\begin{align}
    \mu^*_{\ell} \in \argmax_{\mu \in \mathcal{M}} \int_{[0,1)} \cvar_\alpha^{\ell}(f_{\hat{x}}) d\mu(\alpha), \label{eq:mu_l}
\end{align}
we define the two-sample estimator
\begin{align*}
    z^*_{n,\ell} := \min_{x \in \mathcal{X}} \int_{[0,1)} \cvar_{\alpha}^n(f_x) d\mu^*_{\ell}(\alpha).
\end{align*}
Then, by an analogous argument as for $\hat{z}_{n,m}$, we have $\E^{\mathcal{S}_n}\big[z^*_{n,\ell}\big] \leq z^*$. 

As a consequence, the resulting optimality gap estimator
\begin{align*}
    G_{n,m,\ell} := \hat{z}_{n,m} - z^*_{n,\ell}
\end{align*}
is upward biased. This can be used to construct a probabilistic upper bound for the gap $G$ using the multiple replications procedure \cite{mak1999monte}. Due to the supremum over $\mu \in \mathcal{M}$ in equation~\eqref{eqn:licrm}, Section~\ref{sec:conf_interval}'s risk-neutral reformulation does not hold, so investigating whether approaches such as the one- and two-replications procedure \cite{bayraksan2006assessing} also extend to this case is a topic for further research.


%
In terms of tractability, the integrals in~\eqref{eq:z_hat_nm_coherent} and~\eqref{eq:mu_l} can be treated as in Section~\ref{subsubsec:spectral_tractability}. The maximizations over $\mu \in \mathcal{M}$ in \eqref{eq:z_hat_nm_coherent} and \eqref{eq:mu_l} then simplify, although their tractability depends on the form of $\mathcal{M}$. Investigating this in more detail is also a topic for further research.
%
A final topic for further research is to investigate the quality of $G_{n,m,\ell}$ as a function of the sample sizes $m$ and $\ell$ of the fresh samples used to construct $\hat{z}_{n,m}$ and $z^*_{n,\ell}$, respectively.


\begin{acknowledgements}
The authors thank Andrzej Ruszczy{\'n}ski, Darinka Dentcheva, and Tito Homem-de-Mello for valuable comments on earlier versions of this work.
\end{acknowledgements}

%
%

\bibliographystyle{spmpsci}      
\bibliography{references_dpm}   

\end{document}